\documentclass[11pt]{amsart}
\usepackage[margin=1.15in]{geometry}
\usepackage{amsmath,amssymb,amsthm}
\usepackage[
  colorlinks=true,
  linkcolor=blue,
  citecolor=blue,
  urlcolor=blue,
  pdftitle={Cohomology of amenable traces},
  pdfauthor={Mehdi Moradi},
  pdfsubject={Hochschild cohomology and amenable traces},
  pdfkeywords={amenable trace, Hochschild cohomology, weak expectation, Connes embedding, Christensen-Sinclair cocycle}
]{hyperref}
\usepackage[T1]{fontenc}
\usepackage{lmodern}

\newtheorem{theorem}{Theorem}[section]
\newtheorem{proposition}[theorem]{Proposition}
\newtheorem{corollary}[theorem]{Corollary}
\newtheorem{lemma}[theorem]{Lemma}
\newtheorem*{theoremA}{Theorem A}
\newtheorem*{theoremB}{Theorem B}
\theoremstyle{definition}
\newtheorem{definition}[theorem]{Definition}
\theoremstyle{remark}
\newtheorem{remark}[theorem]{Remark}

\newcommand{\C}{\mathbb C}
\newcommand{\N}{\mathbb N}
\newcommand{\R}{\mathcal R}

\newcommand{\id}{\operatorname{id}}

\newcommand{\tr}{\operatorname{tr}}
\newcommand{\Ad}{\operatorname{Ad}}

\newenvironment{romanenumerate}
  {\begin{enumerate}}
  {\end{enumerate}}

\title[Cohomology of amenable traces]{Cohomology of amenable traces}
\author{Mehdi Moradi}
\address{Department of Mathematics, University of Toronto, Toronto, Ontario, Canada M5S 2E4}
\email{mehdi.moradi@utoronto.ca}

\begin{document}

\begin{abstract}
We construct a bounded Hochschild one-cocycle which detects amenability of
tracial states on unital C*-algebras.  For an arbitrary trace the construction
uses a faithful representation containing the GNS representation as a direct
summand.  When \(\tau\) is faithful, \(H_\tau=L^2(A,\tau)\), and \(J\) is the
canonical conjugation, the cocycle has the particularly simple form
\[
   \delta_\tau(a)(x)
   =J\pi_\tau(a^*)Jx-xJ\pi_\tau(a^*)J.
\]
It is inner in a natural Banach \(A\)-bimodule if and only if \(\tau\) is
amenable.  We also give the corresponding reformulation of embeddability into
an ultrapower of the hyperfinite \(\mathrm{II}_1\) factor.  Finally, for
faithful traces we study a C*-algebraic variant of the Christensen--Sinclair
two-cocycle and prove, with all admissibility details, that amenability makes
this cocycle a coboundary.
\end{abstract}

\maketitle

\section{Introduction}

Injectivity for von Neumann algebras resembles amenability for groups.  Two
classical results of Connes show that a finite von Neumann algebra
\(M\subseteq B(H)\) is injective if and only if there is an \(M\)-central state
on \(B(H)\) extending its trace, and equivalently if and only if suitable
Hochschild cohomology groups vanish
\cite{Connes76,ConnesCohomology}.  Amenable traces are the C*-algebraic
hypertrace analogue of this central-state condition.  One must nevertheless
distinguish amenability of a trace on \(A\) from amenability of the induced
faithful trace on its GNS image: amenability need not pass to quotients.

The purpose of this note is to give a direct cohomological test which retains
the original algebra and therefore works for nonfaithful traces as well.  Let
\(A\) be a unital C*-algebra, let \(\tau\in T(A)\), and let
\[
   \pi_\tau:A\to B(H_\tau),\qquad H_\tau=L^2(A,\tau),
\]
be the GNS representation.  On the dense subspace coming from \(A\), the
canonical conjugation is given by \(J\widehat a=\widehat{a^*}\).  Choose a
faithful unital representation
\[
   \sigma:A\to B(K_0),
\]
put
\[
   K=H_\tau\oplus K_0,
   \qquad
   \rho=\pi_\tau\oplus\sigma:A\to B(K),
\]
and let \(V:H_\tau\to K\) be the inclusion into the first summand.  Thus
\[
   E:B(K)\to B(H_\tau),
   \qquad
   E(T)=V^*TV,
\]
is the compression onto the GNS summand and satisfies
\[
   E(\rho(a)T\rho(b))
   =\pi_\tau(a)E(T)\pi_\tau(b).
\]
For \(a\in A\), set
\[
   R_a=J\pi_\tau(a^*)J.
\]
Define \(X_{\tau,\rho}\) to be the space of bounded linear maps
\(\varphi:B(K)\to B(H_\tau)\) such that
\[
   \varphi(\rho(a))=0
\]
and
\[
   \varphi(\rho(a)T\rho(b))
   =\pi_\tau(a)\varphi(T)\pi_\tau(b)
\]
for all \(a,b\in A\) and \(T\in B(K)\).  Give it the \(A\)-bimodule
structure
\[
   (a\cdot\varphi\cdot b)(T)=R_b\varphi(T)R_a.
\]
The test cocycle is
\[
   \delta_{\tau,\rho}(a)(T)=R_aE(T)-E(T)R_a.
\]
We verify below that \(X_{\tau,\rho}\) is a Banach \(A\)-bimodule and that
\(\delta_{\tau,\rho}:A\to X_{\tau,\rho}\) is a bounded derivation.

The main result is the following.

\begin{theoremA}
The derivation \(\delta_{\tau,\rho}:A\to X_{\tau,\rho}\) is inner if and
only if \(\tau\) is amenable.
\end{theoremA}

The result is independent of the auxiliary faithful representation, since
innerness is equivalent to amenability.  If \(\tau\) is faithful, we may take
\(K=H_\tau\), \(\rho=\pi_\tau\), and \(E=\id\); this gives exactly the
GNS-only formula displayed in the abstract.  Theorem \ref{thm:main} and
Corollary \ref{cor:faithful-main} give the numbered statements.  As an
application, we obtain a cohomological reformulation of embeddability into an
ultrapower of the hyperfinite \(\mathrm{II}_1\) factor.

\begin{theoremB}
Let \(M\) be a \(\mathrm{II}_1\) factor with separable predual and faithful
normal trace \(\tau\).  Then \(M\) embeds trace-preservingly into
\(\R^\omega\) if and only if there is a separable unital ultraweakly dense
C*-subalgebra \(A\subseteq M\) such that the faithful-trace test cocycle
\(\delta_{\tau|_A}\) is inner.
\end{theoremB}

The precise form appears as Theorem \ref{thm:embeddable}.  We also include a
C*-algebraic variant of the Christensen--Sinclair two-cocycle associated to a
faithful trace.  The implication from amenability to coboundary follows from
Brown's weak-expectation criterion.  The original von Neumann algebra converse
does not automatically transfer to this larger C*-algebraic coefficient
module; Remark \ref{rem:cs-converse} identifies the obstruction precisely.

\section{Preliminaries on Hochschild cohomology}

Let \(A\) be a C*-algebra and let \(X\) be a Banach \(A\)-bimodule. We write \(L^n(A,X)\) for the Banach space of continuous \(n\)-linear maps \(A^n\to X\). The Hochschild coboundary map
\[
   \Delta^n:L^n(A,X)\to L^{n+1}(A,X)
\]
is given by
\begin{align*}
\Delta^n\varphi(a_1,\ldots,a_{n+1})
 &=a_1\varphi(a_2,\ldots,a_{n+1}) \\
 &\quad +\sum_{j=1}^n(-1)^j\varphi(a_1,\ldots,a_{j-1},a_ja_{j+1},a_{j+2},\ldots,a_{n+1}) \\
 &\quad +(-1)^{n+1}\varphi(a_1,\ldots,a_n)a_{n+1}.
\end{align*}
We write
\[
   Z^n(A,X)=\ker\Delta^n,
   \qquad B^n(A,X)=\operatorname{im}\Delta^{n-1},
\]
and
\[
   H^n(A,X)=Z^n(A,X)/B^n(A,X).
\]
For \(x\in X\), the zero-dimensional coboundary is
\[
   (\Delta^0x)(a)=a\cdot x-x\cdot a.
\]
Thus a derivation is \emph{inner} precisely when it belongs to \(B^1(A,X)\).
With the alternative commutator convention
\(\operatorname{ad}(x)(a)=x\cdot a-a\cdot x\), one has
\(\Delta^0x=\operatorname{ad}(-x)\); the two sign conventions therefore give
the same space of inner derivations.

We shall use the following elementary exact sequence.

\begin{proposition}\label{prop:exact}
Let \(A\subseteq E\) be a unital inclusion of C*-algebras, and let
\[
   0\longrightarrow I\longrightarrow E\xrightarrow q B\longrightarrow 0
\]
be a short exact sequence. Then there is an exact sequence
\[
0\longrightarrow q(A'\cap E)\longrightarrow q(A)'\cap B
   \xrightarrow{\partial} H^1(A,I)\longrightarrow H^1(A,E)\longrightarrow H^1(A,B).
\]
\end{proposition}

\begin{proof}
Regard \(I\), \(E\), and \(B\) as \(A\)-bimodules in the natural way,
using \(q|_A\) for the action on \(B\).  The maps between the cohomology
groups are induced by \(I\hookrightarrow E\) and by \(q\), respectively.

Let \(b\in q(A)'\cap B\), and choose \(e\in E\) with \(q(e)=b\).  Since
\(b\) commutes with \(q(A)\),
\[
   q([e,a])=[q(e),q(a)]=0,
\]
so \(D_e(a)=[e,a]\) is an \(I\)-valued derivation.  Define
\[
   \partial(b)=[D_e]\in H^1(A,I).
\]
If \(e'\) is another lift, then \(e-e'\in I\) and \(D_e-D_{e'}\) is
inner as an \(I\)-valued derivation.  Thus \(\partial\) is well-defined.

The first arrow is the inclusion
\(q(A'\cap E)\subseteq q(A)'\cap B\).  If \(b=q(e)\) with
\(e\in A'\cap E\), then \(D_e=0\).  Conversely, if
\(\partial(b)=0\), there is \(i\in I\) such that
\([e,a]=[i,a]\) for every \(a\in A\).  Hence \(e-i\in A'\cap E\) and
\(b=q(e-i)\).  This proves exactness at \(q(A)'\cap B\).

If \(b=q(e)\in q(A)'\cap B\), then \(D_e\), regarded as an \(E\)-valued
derivation, is inner.  Conversely, suppose that \(D:A\to I\) becomes inner
as an \(E\)-valued derivation.  Thus \(D(a)=[e,a]\) for some \(e\in E\).
Because \(D(A)\subseteq I\), the element \(q(e)\) commutes with \(q(A)\),
and \([D]=\partial(q(e))\).  This proves exactness at \(H^1(A,I)\).

Finally, the image of \(H^1(A,I)\) is contained in the kernel of
\(H^1(A,E)\to H^1(A,B)\).  Conversely, let \(D:A\to E\) be a derivation
such that \(q\circ D\) is inner.  Choose \(b\in B\) with
\[
   q(D(a))=[b,q(a)]
\]
and lift \(b\) to \(e\in E\).  Then
\[
   D_0(a)=D(a)-[e,a]
\]
is an \(I\)-valued derivation, and \(D\) and \(D_0\) determine the same
class in \(H^1(A,E)\).  This proves exactness at \(H^1(A,E)\).
\end{proof}

\begin{corollary}\label{cor:calkin}
Let \(A\subseteq B(H)\) be a unital C*-algebra whose given representation
has a cyclic vector, and let
\[
   q:B(H)\to \mathcal Q(H):=B(H)/K(H)
\]
be the quotient map. Then
\[
0\longrightarrow q(A')\longrightarrow q(A)'\cap \mathcal Q(H)
   \xrightarrow{\partial}H^1(A,K(H))\longrightarrow 0
\]
is exact, where \(\partial(q(T))=[a\mapsto[T,a]]\).
\end{corollary}

\begin{proof}
Apply Proposition \ref{prop:exact} to
\[
   0\longrightarrow K(H)\longrightarrow B(H)\longrightarrow \mathcal Q(H)\longrightarrow 0.
\]
Christensen proved that every bounded derivation \(A\to B(H)\) is spatial
when the given representation has a cyclic vector; equivalently,
\(H^1(A,B(H))=0\) \cite[Corollary~5.4(b)]{Christensen82}.  The asserted
short exact sequence follows.
\end{proof}

In particular, if \(D:A\to K(H)\) has the form \(D(a)=[T,a]\) with
\(T\in B(H)\), then \(q(T)\in q(A)'\cap\mathcal Q(H)\), and \(D\) has a
compact implementer if and only if \(q(T)\in q(A')\).

\section{Approximation properties of traces}

We recall the trace approximation properties used below.  Nets may be replaced
by sequences when \(A\) is separable.

\begin{definition}\label{def:traceclasses}
Let \(A\) be a unital C*-algebra and let \(\tau\in T(A)\).
\begin{romanenumerate}
\item \(\tau\) is \emph{hyperlinear} if there is a trace-preserving embedding
\[
   \pi_\tau(A)''\hookrightarrow \R^\omega
\]
for some free ultrafilter \(\omega\), where \(\R\) is the hyperfinite
\(\mathrm{II}_1\) factor.
\item \(\tau\) is \emph{amenable} if there is a net of u.c.p. maps
\[
   \varphi_i:A\to M_{n(i)}
\]
such that, for all \(a,b\in A\),
\[
   \tr_{n(i)}(\varphi_i(a))\to \tau(a),
   \qquad
   \|\varphi_i(ab)-\varphi_i(a)\varphi_i(b)\|_{2,\tr_{n(i)}}\to 0.
\]
\item \(\tau\) is \emph{quasidiagonal} if the maps in (ii) can be chosen so that
\[
   \|\varphi_i(ab)-\varphi_i(a)\varphi_i(b)\|\to 0
\]
in operator norm for all \(a,b\in A\).
\item \(\tau\) is \emph{locally finite-dimensional}, or LFD, if there is a
net of u.c.p. maps
\[
   \varphi_i:A\to M_{n(i)}
\]
such that \(\tr_{n(i)}(\varphi_i(a))\to \tau(a)\) for all \(a\in A\), and
\[
   d(a,\mathcal M_{\varphi_i})\to 0
\]
for all \(a\in A\), where \(\mathcal M_{\varphi_i}\) is the multiplicative domain of \(\varphi_i\).
\item \(\tau\) is \emph{residually finite-dimensional}, or RFD, if it is a
weak-* limit of traces of the form \(\tr_d\circ\theta\), where
\(\theta:A\to M_d\) is a unital finite-dimensional representation.
\end{romanenumerate}
\end{definition}

\begin{remark}
On \(C^*(\mathbb F_n)\), where \(n=1,2,\ldots,\infty\), all five notions in
Definition \ref{def:traceclasses} coincide
\cite[Propositions~4.1.14 and~6.3.4]{Brown06}.  On the full group
C*-algebra of a discrete property~(T) group, amenable, quasidiagonal, LFD,
and RFD traces coincide
\cite[Lemma~4.1.11 and Proposition~4.1.12]{Brown06}.
\end{remark}

The faithful-representation clause in the next theorem is essential.  The GNS
representation itself cannot be used here when the trace is nonfaithful.

\begin{theorem}[Brown]\label{thm:brown}
Let \(A\) be a unital C*-algebra, let \(\tau\in T(A)\), and let
\(\rho:A\to B(K)\) be a faithful unital representation.  The following are
equivalent \cite[Theorem~3.1.6]{Brown06}.
\begin{romanenumerate}
\item \(\tau\) is amenable.
\item There is a state \(\Omega\) on \(B(K)\) such that
\[
   \Omega(\rho(a))=\tau(a),
   \qquad
   \Omega(\rho(a)T)=\Omega(T\rho(a))
\]
for all \(a\in A\) and \(T\in B(K)\).
\item There is a u.c.p. map
\[
   u:B(K)\to \pi_\tau(A)''
\]
such that \(u(\rho(a))=\pi_\tau(a)\) for all \(a\in A\).
\end{romanenumerate}
\end{theorem}

We need the following bounded form.  Condition (iii) is a mixed bimodule
condition: the copy \(\rho(A)\) acts in the domain and the GNS copy
\(\pi_\tau(A)\) acts in the codomain.

\begin{proposition}\label{prop:bounded}
Under the hypotheses of Theorem \ref{thm:brown}, the following are equivalent.
\begin{romanenumerate}
\item \(\tau\) is amenable.
\item There is a bounded functional \(\Phi\in B(K)^*\) satisfying
\[
   \Phi(\rho(a))=\tau(a),
   \qquad
   \Phi(\rho(a)T)=\Phi(T\rho(a))
\]
for all \(a\in A\) and \(T\in B(K)\).
\item There is a bounded linear map
\[
   u:B(K)\to \pi_\tau(A)''
\]
such that \(u(\rho(a))=\pi_\tau(a)\) and
\[
   u(\rho(a)T\rho(b))
   =\pi_\tau(a)u(T)\pi_\tau(b)
\]
for all \(a,b\in A\) and \(T\in B(K)\).
\end{romanenumerate}
\end{proposition}

\begin{proof}
For (i)\(\Rightarrow\)(iii), take the u.c.p. map supplied by Theorem
\ref{thm:brown}.  For \(a\in A\),
\[
 u(\rho(a)^*\rho(a))
 =\pi_\tau(a)^*\pi_\tau(a)
 =u(\rho(a))^*u(\rho(a)),
\]
and the analogous equality holds with the factors reversed.  Hence
\(\rho(A)\) lies in the multiplicative domain of \(u\), which gives the
mixed bimodule identity.

For (iii)\(\Rightarrow\)(ii), define
\[
   \Phi(T)=\widetilde\tau(u(T)),
\]
where \(\widetilde\tau\) is the normal trace on \(\pi_\tau(A)''\).  Then
\(\Phi(\rho(a))=\tau(a)\), and bimodularity and traciality give
\[
   \Phi(\rho(a)T)
   =\widetilde\tau(\pi_\tau(a)u(T))
   =\widetilde\tau(u(T)\pi_\tau(a))
   =\Phi(T\rho(a)).
\]

Assume (ii).  Put
\[
   \Phi^\sharp(T)=\overline{\Phi(T^*)},
   \qquad
   f=\frac{\Phi+\Phi^\sharp}{2}.
\]
Then \(f\) is self-adjoint, \(\rho(A)\)-central, and
\(f(\rho(a))=\tau(a)\).  Write its Jordan decomposition as
\(f=f_+-f_-\).  Centrality makes \(f\) invariant under
\(\Ad(\rho(v))\) for every unitary \(v\in A\).  Uniqueness of the Jordan
decomposition makes \(f_+\) and \(f_-\) invariant as well.  For such a
unitary,
\[
   f_+(\rho(v)T)
   =f_+(\rho(v)(T\rho(v))\rho(v)^*)
   =f_+(T\rho(v)),
\]
and similarly for \(f_-\).  Since unitaries linearly span a unital
C*-algebra, both positive parts are \(\rho(A)\)-central.

Let \(c=f_+(1)\).  On \(\rho(A)\) we have
\[
   f_+=\tau+f_-,
\]
so \(c\geq1\).  The functional \(\Omega=c^{-1}f_+\) is a
\(\rho(A)\)-central state.  Thus
\[
   \gamma(a)=\Omega(\rho(a))
\]
is an amenable tracial state by Theorem \ref{thm:brown}.  If \(c=1\),
then \(f_-=0\) and \(\gamma=\tau\).  If \(c>1\), then
\[
   \eta(a)=\frac{f_-(\rho(a))}{c-1}
\]
is a tracial state and
\[
   \gamma=\frac1c\tau+\left(1-\frac1c\right)\eta.
\]
Amenable traces form a face of \(T(A)\)
\cite[Proposition~3.5.3]{Brown06}; hence \(\tau\) is amenable.
\end{proof}

\section{The test one-cocycle}

Retain the notation \(K\), \(\rho\), \(E\), \(R_a\), and
\(X_{\tau,\rho}\) from the introduction.

\begin{lemma}\label{lem:testmodule}
The space \(X_{\tau,\rho}\) is a norm-closed subspace of
\(\mathcal B(B(K),B(H_\tau))\), hence is a Banach space.  The formula
\[
   (a\cdot\varphi\cdot b)(T)=R_b\varphi(T)R_a
\]
makes it a unital Banach \(A\)-bimodule.  Moreover,
\(\delta_{\tau,\rho}(a)\in X_{\tau,\rho}\), and
\(\delta_{\tau,\rho}:A\to X_{\tau,\rho}\) is a bounded derivation with
\[
   \|\delta_{\tau,\rho}(a)\|
   \leq2\|\pi_\tau(a)\|
   \leq2\|a\|.
\]
\end{lemma}

\begin{proof}
The tracial identity
\[
 \|\widehat{a^*}\|_2^2=\tau(aa^*)=\tau(a^*a)=\|\widehat a\|_2^2
\]
shows directly that \(J\widehat a=\widehat{a^*}\) defines an antiunitary
involution.  On the dense subspace coming from \(A\), \(R_a\) is right
multiplication by \(a\).  Consequently,
\[
   R_{ab}=R_bR_a,
   \qquad
   R_a\pi_\tau(c)=\pi_\tau(c)R_a,
   \qquad
   \|R_a\|=\|\pi_\tau(a)\|.
\]
The standard commutant relation for the tracial GNS representation is
\[
   (J\pi_\tau(A)J)'=\pi_\tau(A)''.
\]

Norm limits preserve the vanishing and mixed bimodularity conditions defining
\(X_{\tau,\rho}\), so this space is closed.  Since each \(R_a\) commutes
with \(\pi_\tau(A)\), multiplication of the values of \(\varphi\) by
\(R_b\) and \(R_a\) preserves those conditions.  The identities
\(R_{ab}=R_bR_a\) give the module axioms, and
\[
   \|a\cdot\varphi\cdot b\|
   \leq\|a\|\,\|\varphi\|\,\|b\|.
\]

Because \(E(\rho(c))=\pi_\tau(c)\), the operator
\(\delta_{\tau,\rho}(a)(\rho(c))\) is zero.  Also, for
\(c,d\in A\),
\begin{align*}
 \delta_{\tau,\rho}(a)(\rho(c)T\rho(d))
 &=R_a\pi_\tau(c)E(T)\pi_\tau(d)
   -\pi_\tau(c)E(T)\pi_\tau(d)R_a\\
 &=\pi_\tau(c)\delta_{\tau,\rho}(a)(T)\pi_\tau(d).
\end{align*}
Thus \(\delta_{\tau,\rho}(a)\in X_{\tau,\rho}\).  Finally,
\begin{align*}
 &\bigl(a\cdot\delta_{\tau,\rho}(b)
       +\delta_{\tau,\rho}(a)\cdot b\bigr)(T)\\
 &\quad=(R_bE(T)-E(T)R_b)R_a
       +R_b(R_aE(T)-E(T)R_a)\\
 &\quad=R_bR_aE(T)-E(T)R_bR_a\\
 &\quad=\delta_{\tau,\rho}(ab)(T).
\end{align*}
Since \(E\) is contractive, the norm estimate is immediate.
\end{proof}

\begin{theorem}\label{thm:main}
Let \(A\) be a unital C*-algebra and let \(\tau\in T(A)\).  Choose a
faithful representation \(\rho\) containing \(\pi_\tau\) as a direct
summand, as above.  Then
\[
   \delta_{\tau,\rho}:A\to X_{\tau,\rho}
\]
is inner if and only if \(\tau\) is amenable.
\end{theorem}

\begin{proof}
Suppose first that \(\tau\) is amenable.  Since \(\rho\) is faithful,
Theorem \ref{thm:brown} gives a u.c.p. map
\[
   u:B(K)\to\pi_\tau(A)''
\]
such that \(u(\rho(a))=\pi_\tau(a)\).  The multiplicative-domain theorem
makes \(u\) a mixed \(A\)-bimodule map.  The compression \(E\) has the
same mixed bimodularity, so
\[
   \varphi=u-E
\]
belongs to \(X_{\tau,\rho}\).  Since the range of \(u\) commutes with
every \(R_a\),
\begin{align*}
 (\Delta^0\varphi)(a)(T)
 &=\varphi(T)R_a-R_a\varphi(T)\\
 &=(u(T)-E(T))R_a-R_a(u(T)-E(T))\\
 &=R_aE(T)-E(T)R_a\\
 &=\delta_{\tau,\rho}(a)(T).
\end{align*}
Thus the test cocycle is inner.

Conversely, suppose \(\delta_{\tau,\rho}=\Delta^0\varphi\) for some
\(\varphi\in X_{\tau,\rho}\).  Then
\[
   R_aE(T)-E(T)R_a=\varphi(T)R_a-R_a\varphi(T),
\]
or equivalently,
\[
   R_a(E(T)+\varphi(T))=(E(T)+\varphi(T))R_a.
\]
Hence \(U=E+\varphi\) takes values in
\((J\pi_\tau(A)J)'=\pi_\tau(A)''\).  It is a bounded mixed
\(A\)-bimodule map and satisfies \(U(\rho(a))=\pi_\tau(a)\).
Proposition \ref{prop:bounded} now implies that \(\tau\) is amenable.
\end{proof}

\begin{corollary}\label{cor:faithful-main}
Suppose that \(\tau\) is faithful.  Let \(X_\tau\) be the Banach space of
bounded \(\pi_\tau(A)\)-bimodule maps
\(\varphi:B(H_\tau)\to B(H_\tau)\) which vanish on \(\pi_\tau(A)\),
with
\[
   (a\cdot\varphi\cdot b)(T)=R_b\varphi(T)R_a.
\]
Then the derivation
\[
   \delta_\tau(a)(T)=R_aT-TR_a
\]
is inner if and only if \(\tau\) is amenable.
\end{corollary}

\begin{proof}
Because \(\tau\) is faithful, \(\pi_\tau\) is faithful.  In Theorem
\ref{thm:main}, take \(K=H_\tau\), \(\rho=\pi_\tau\), and \(E=\id\).
\end{proof}

\begin{remark}\label{rem:nonfaithful}
The faithful summand in Theorem \ref{thm:main} cannot simply be omitted.  For
an arbitrary trace, the GNS-only cocycle in Corollary
\ref{cor:faithful-main} depends only on \(B_\tau=\pi_\tau(A)\) and detects
amenability of the induced faithful trace
\(\bar\tau(\pi_\tau(a))=\tau(a)\) on \(B_\tau\).  Amenability need not
descend from \(A\) to this quotient.

For example, let \(A=C^*(\mathbb F_2)\) and let \(\tau\) be the canonical
trace.  Residual finiteness of \(\mathbb F_2\) makes \(\tau\) a weak-* limit of
finite-dimensional traces, hence amenable
\cite[Proposition~4.1.4]{Brown06}.  But
\(B_\tau=C_r^*(\mathbb F_2)\), whose canonical trace is not amenable because
\(\mathbb F_2\) is nonamenable \cite[Proposition~4.1.1]{Brown06}.  Thus the
GNS-only cocycle is not inner in this example, even though \(\tau\) is
amenable on the full group C*-algebra.
\end{remark}

\section{Embeddable \texorpdfstring{\(\mathrm{II}_1\)}{II1} factors}

\begin{theorem}\label{thm:embeddable}
Let \(M\) be a \(\mathrm{II}_1\) factor with separable predual and faithful
normal trace \(\tau\).  The following are equivalent.
\begin{romanenumerate}
\item There is a unital normal trace-preserving *-monomorphism
\[
   M\longrightarrow\R^\omega
\]
for some free ultrafilter \(\omega\) on \(\N\).
\item There are separable unital C*-algebras \(B\) and \(A\subseteq M\),
with \(A\) ultraweakly dense in \(M\), and a surjective *-homomorphism
\(\rho:B\to A\) such that \(\tau|_A\circ\rho\) is amenable.
\item There is a unital *-monomorphism
\[
   \rho:C^*(\mathbb F_\infty)\longrightarrow M
\]
with ultraweakly dense range such that \(\tau\circ\rho\) is amenable.
Equivalently, \(\tau\circ\rho\) is quasidiagonal, LFD, or RFD.
\item There is a separable unital ultraweakly dense C*-subalgebra
\(A\subseteq M\) such that \(\tau|_A\) is amenable.
\item There is a separable unital ultraweakly dense C*-subalgebra
\(A\subseteq M\) such that the faithful-trace cocycle
\(\delta_{\tau|_A}\) of Corollary \ref{cor:faithful-main} is inner.
\end{romanenumerate}
\end{theorem}

\begin{proof}
Assume (ii), and put \(\sigma=\tau|_A\circ\rho\).  Faithfulness of
\(\tau\) gives
\[
   \{b\in B:\sigma(b^*b)=0\}=\ker\rho.
\]
Consequently, the map initially defined by
\[
   W\widehat b=\widehat{\rho(b)}
\]
identifies \(L^2(B,\sigma)\) with \(L^2(M,\tau)\).  Its range is dense:
by Kaplansky density, the ultraweak density of \(A\) implies its \(2\)-norm
density on bounded subsets of \(M\).  Under this identification,
\[
   \pi_\sigma(B)''\cong A''=M.
\]
The finite-dimensional approximants for the amenable trace \(\sigma\) give a
trace-preserving embedding of its GNS von Neumann algebra into \(\R^\omega\)
\cite[Theorems~3.1.6 and~3.1.7]{Brown06}.  Hence (ii) implies (i).

Now assume (i).  Brown's universal representation result gives a unital
*-monomorphism
\[
   \rho:C^*(\mathbb F_\infty)\longrightarrow M
\]
with ultraweakly dense range \cite[Proposition~5.1.1]{Brown06}.  Put
\(\sigma=\tau\circ\rho\).  Its GNS von Neumann algebra is \(M\), so (i)
says that \(\sigma\) is hyperlinear.  The algebra
\(C^*(\mathbb F_\infty)\) has the local lifting property, and therefore
hyperlinearity of \(\sigma\) implies amenability
\cite[Proposition~6.3.4]{Brown06}.  This proves (iii).  The equivalence with
the other approximation properties follows from
\cite[Proposition~4.1.14]{Brown06}.

Condition (iii) implies (iv): take
\(A=\rho(C^*(\mathbb F_\infty))\).  The map \(\rho\) is a
*-isomorphism onto \(A\), so its finite-dimensional approximants transfer to
\(\tau|_A\).  Condition (iv) implies (ii) by taking \(B=A\) and
\(\rho=\id_A\).  Finally, \(\tau|_A\) is faithful for every C*-subalgebra
\(A\subseteq M\), so (iv) is equivalent to (v) by Corollary
\ref{cor:faithful-main}.
\end{proof}

\section{A C*-algebraic Christensen--Sinclair two-cocycle}

This section adapts the cocycle of Christensen and Sinclair
\cite[\S\S2.1--2.3]{CS97} to a faithful tracial C*-algebra.  Their von
Neumann algebra coefficient module requires vanishing on
\(M'\times B(H)\); the module below requires vanishing only on
\(JAJ\times B(H)\).  This distinction is essential because Brown's weak
expectation fixes \(JAJ\), but need not fix all of \(M'\).

Let \(\tau\) be a faithful tracial state on a unital C*-algebra \(A\), and
put
\[
   H=L^2(A,\tau),
   \qquad
   M=\pi_\tau(A)'',
   \qquad
   M'=JMJ.
\]
We suppress \(\pi_\tau\) from the notation, so \(JAJ\subseteq M'\).

For a bilinear map \(F:B(H)\times B(H)\to B(H)\), define
\[
 F^{(n)}(X,Y)_{ij}
   =\sum_{k=1}^nF(x_{ik},y_{kj}),
 \qquad
 X=(x_{ij}),\quad Y=(y_{ij})\in M_n(B(H)),
\]
and set
\[
   \|F\|_{\mathrm{cb}}=\sup_{n\geq1}\|F^{(n)}\|.
\]
Complete boundedness of bilinear maps in this section always refers to this
Christensen--Sinclair convention.

\begin{definition}
A completely bounded bilinear map
\(F:B(H)\times B(H)\to B(H)\) is \emph{admissible} if
\begin{romanenumerate}
\item \(F|_{JAJ\times B(H)}=0\);
\item \(F|_{B(H)\times\C1}=0\);
\item \(F(B(H),K(H))\subseteq K(H)\);
\item for every \(x\in B(H)\), the map \(y\mapsto F(x,y)\) is
ultraweakly continuous.
\end{romanenumerate}
Let \(S_\tau\) be the space of admissible bilinear maps, equipped with the
completely bounded norm.
\end{definition}

\begin{lemma}\label{lem:stau}
The space \(S_\tau\) is a closed subspace of the completely bounded
bilinear maps and is a Banach \(A\)-bimodule under
\[
   (a\cdot F\cdot b)(x,y)=aF(x,y)b.
\]
Moreover,
\[
   \|a\cdot F\cdot b\|_{\mathrm{cb}}
   \leq\|a\|\,\|F\|_{\mathrm{cb}}\,\|b\|.
\]
\end{lemma}

\begin{proof}
Conditions (i) and (ii) are closed under norm convergence, and condition
(iii) is closed because \(K(H)\) is norm closed.  Suppose that
\(F_j\to F\) in completely bounded norm.  For fixed \(x\in B(H)\), the
normal maps \(y\mapsto F_j(x,y)\) converge in operator norm to
\(y\mapsto F(x,y)\).  Normal maps between dual Banach spaces form a
norm-closed subspace: their preadjoints converge in operator norm.  Thus the
limit is normal, and condition (iv) is closed as well.

All four conditions are preserved by multiplying the value on the left and
right by elements of \(A\), since \(K(H)\) is a two-sided ideal and
multiplication by a fixed operator is ultraweakly continuous.  Finally,
\[
 (a\cdot F\cdot b)^{(n)}(X,Y)
  =(1_n\otimes a)F^{(n)}(X,Y)(1_n\otimes b),
\]
which proves the norm estimate.
\end{proof}

\begin{lemma}\label{lem:cs-cocycle}
The formula
\[
   \Phi_{\mathrm{cs}}(a,b)(x,y)=[a,x][b,y]
\]
defines a bounded bilinear map \(\Phi_{\mathrm{cs}}:A\times A\to S_\tau\)
such that
\[
   \|\Phi_{\mathrm{cs}}(a,b)\|_{\mathrm{cb}}
   \leq4\|a\|\,\|b\|.
\]
Moreover, \(\Phi_{\mathrm{cs}}\in Z^2(A,S_\tau)\).
\end{lemma}

\begin{proof}
For \(a\in A\), let \(d_a(x)=[a,x]\).  Then
\(\|d_a\|_{\mathrm{cb}}\leq2\|a\|\), and
\[
   \Phi_{\mathrm{cs}}(a,b)^{(n)}(X,Y)
   =d_a^{(n)}(X)d_b^{(n)}(Y).
\]
This proves the stated completely bounded norm estimate.

If \(x\in JAJ\), then \(x\) commutes with \(A\), so the value is zero.
The value is also zero when \(y\in\C1\).  If \(y\in K(H)\), then
\([b,y]\in K(H)\), and hence \([a,x][b,y]\in K(H)\).  Finally,
\(y\mapsto[a,x][b,y]\) is ultraweakly continuous.  Thus the value lies in
\(S_\tau\).

For \(a,b,c\in A\), use
\[
   [ab,x]=a[b,x]+[a,x]b,
   \qquad
   [bc,y]=b[c,y]+[b,y]c.
\]
Then
\begin{align*}
 (\Delta^2\Phi_{\mathrm{cs}})(a,b,c)(x,y)
 &=a[b,x][c,y]-[ab,x][c,y]\\
 &\quad+[a,x][bc,y]-[a,x][b,y]c\\
 &=0.
\end{align*}
Hence \(\Phi_{\mathrm{cs}}\) is a continuous Hochschild two-cocycle.
\end{proof}

\begin{proposition}\label{prop:cs}
If \(\tau\) is amenable, then
\[
   \Phi_{\mathrm{cs}}\in B^2(A,S_\tau).
\]
More precisely, it is the coboundary of a one-cochain
\(\varphi:A\to S_\tau\) satisfying \(\|\varphi(a)\|_{\mathrm{cb}}\leq
4\|a\|\).
\end{proposition}

\begin{proof}
Let \(A^{\mathrm{op}}\) be the opposite algebra and define
\(\tau^{\mathrm{op}}(a^{\mathrm{op}})=\tau(a)\).  Amenability passes to the
opposite algebra.  Indeed, if \(\theta_i:A\to M_{n(i)}\) witnesses
amenability of \(\tau\), then
\[
   \theta_i^{\mathrm{op}}(a^{\mathrm{op}})=\theta_i(a)^{\mathsf t}
\]
witnesses amenability of \(\tau^{\mathrm{op}}\), using the canonical
*-isomorphism \(M_{n(i)}^{\mathrm{op}}\cong M_{n(i)}\).  For example,
\begin{align*}
 &\|\theta_i^{\mathrm{op}}(a^{\mathrm{op}}b^{\mathrm{op}})
     -\theta_i^{\mathrm{op}}(a^{\mathrm{op}})
      \theta_i^{\mathrm{op}}(b^{\mathrm{op}})\|_2\\
 &\qquad=\|\theta_i(ba)-\theta_i(b)\theta_i(a)\|_2\longrightarrow0.
\end{align*}

The unitary
\[
   U_0:H_{\tau^{\mathrm{op}}}\to H,
   \qquad
   U_0\widehat{a^{\mathrm{op}}}=\widehat a,
\]
identifies the opposite GNS representation with the right representation:
\[
   U_0\pi_{\tau^{\mathrm{op}}}(a^{\mathrm{op}})U_0^*=Ja^*J.
\]
Its von Neumann closure is \((JAJ)''=M'\).  Since \(\tau\) is faithful,
this right representation is faithful.  Theorem \ref{thm:brown}, applied to
\((A^{\mathrm{op}},\tau^{\mathrm{op}})\), therefore gives a u.c.p. map
\[
   u:B(H)\to M'
\]
which fixes \(JAJ\).  The multiplicative-domain theorem makes \(u\)
\(JAJ\)-bimodular.

Define
\[
   \varphi(a)(x,y)=(x-u(x))[a,y].
\]
Let \(T=\id-u\).  Since \(u\) is u.c.p.,
\[
   \|T\|_{\mathrm{cb}}\leq2,
   \qquad
   \|d_a\|_{\mathrm{cb}}\leq2\|a\|,
\]
and, at matrix level,
\[
   \varphi(a)^{(n)}(X,Y)=T^{(n)}(X)d_a^{(n)}(Y).
\]
Thus \(\|\varphi(a)\|_{\mathrm{cb}}\leq4\|a\|\), so
\(a\mapsto\varphi(a)\) is bounded and linear.

If \(x\in JAJ\), then \(u(x)=x\).  If \(y\in\C1\), then
\([a,y]=0\).  If \(y\in K(H)\), then \([a,y]\in K(H)\), so
\((x-u(x))[a,y]\in K(H)\).  Finally,
\(y\mapsto(x-u(x))[a,y]\) is ultraweakly continuous.  Hence
\(\varphi(a)\in S_\tau\).

Put \(z_x=x-u(x)\).  Since \(u(x)\in M'\), it commutes with \(A\).
Using \([ab,y]=a[b,y]+[a,y]b\), we get
\begin{align*}
 (\Delta^1\varphi)(a,b)(x,y)
 &=az_x[b,y]-z_x[ab,y]+z_x[a,y]b\\
 &=az_x[b,y]-z_xa[b,y]\\
 &=[a,z_x][b,y]\\
 &=[a,x][b,y]\\
 &=\Phi_{\mathrm{cs}}(a,b)(x,y).
\end{align*}
Therefore \(\Phi_{\mathrm{cs}}=\Delta^1\varphi\).
\end{proof}

\begin{remark}\label{rem:cs-converse}
No converse is asserted here.  The von Neumann algebra module in
\cite[\S2]{CS97} requires vanishing on \(M'\times B(H)\), whereas the
present, generally larger, module requires vanishing only on
\(JAJ\times B(H)\).  Moreover, the converse proof in
\cite[Theorem~2.3]{CS97} uses the fact that, on the continuous von Neumann
summand, the relevant compact-valued derivations have unique compact
implementers with uniform norm estimates.

For the present cyclic C*-representation, Christensen's theorem guarantees
only an implementer in \(B(H)\).  If \(D:A\to K(H)\) is written as
\(D(a)=[T,a]\), Corollary \ref{cor:calkin} identifies the obstruction to a
compact implementer with the coset of
\[
   q(T)\in q(A)'\cap\mathcal Q(H)
\]
modulo \(q(A')\).  Nothing in the present hypotheses forces this coset to
vanish.  Even if all pointwise obstructions vanish, a converse would still
require a simultaneous linear, completely bounded choice of implementers.
Thus the Christensen--Sinclair converse does not presently transfer to this
C*-algebraic coefficient module.
\end{remark}

\section*{Acknowledgements}
The author gratefully acknowledges the support and hospitality of the Fields Institute for Research in Mathematical Sciences, where he was a postdoctoral fellow for six months during the Operator Algebras Thematic Program. He also thanks the Department of Mathematics and Statistics at the University of Ottawa for its support and hospitality during his postdoctoral fellowship.

During the preparation of this manuscript, GPT-5.5 Pro was used to assist with
drafting and editorial refinement.  The author independently checked,
substantially revised, and approved all content, and takes full responsibility
for the final version.

\end{document}